\documentclass[11pt]{article}

\usepackage{amssymb,amsmath, amsfonts, amsthm}

\usepackage{graphicx}
\usepackage{pgfplots}
\usepackage{tikz}
\usepackage{booktabs}
\usepackage{array}
\usepackage{verbatim}
\usepackage{subfig}
\usepackage{geometry}
\usepackage{scrextend}

\usepackage{fancyhdr}

\newcommand{\abs}[1]{\left\vert#1\right\vert}
\newcommand{\vertex}{\node[vertex]}
\tikzstyle{vertex}=[circle, draw, inner sep=0pt, minimum size=6pt]

\newtheorem{theorem}{Theorem}
\newtheorem{lemma}{Lemma}
\newtheorem{corollary}{Corollary}
\newtheorem{proposition}{Proposition}

\newtheorem{observation}{Observation}

\newcommand{\spider}{\hskip1pt \textrm{sp} \hskip1pt}

\begin{document}
\title{Graphs which satisfy a Vizing-like bound for power domination of Cartesian products}
\author{
Sarah E. Anderson$^{a}$
\and
Kirsti Kuenzel$^{b}$
\and
Houston Schuerger$^{b}$
}

\date{\today}

\maketitle

\begin{center}
$^a$ Department of Mathematics, University of St. Thomas, St. Paul, MN 55105\\
$^b$ Department of Mathematics, Trinity College, Hartford, CT 06106\\

\end{center}
\vskip15mm
\begin{abstract} 
Power domination is a two-step observation process that is used to monitor power networks and can be viewed as a combination of domination and zero forcing. Given a graph $G$, a subset $S\subseteq V(G)$  that can observe all vertices of $G$ using this process is known as a power dominating set of $G$, and the power domination number of $G$, $\gamma_P(G)$, is the minimum number of vertices in a power dominating set. We introduce a new partition on the vertices of a graph to provide a lower bound for the power domination number. We also consider the power domination number of the Cartesian product of two graphs, $G \Box H$, and show certain graphs satisfy a Vizing-like bound with regards to the power domination number. In particular, we prove that for any two trees $T_1$ and $T_2$, $\gamma_P(T_1)\gamma_P(T_2) \leq \gamma_P(T_1 \Box T_2)$.
\end{abstract}

\section{Introduction} \label{sec:intro}
Electrical power networks may be monitored through a  two-step process called power domination. Phase measurement units (PMUs) are used to observe the information; however, since PMUs are expensive, one would like to use the smallest number of PMUs possible to observe the power network. Haynes et al. first translated this problem to a graph in \cite{electricgrids}. Consider the graph $G = (V(G), E(G))$, where the vertices represent the electrical nodes and the edges represent transmission lines between two electrical nodes. Due to this context, in this paper, we will consider simple, finite graphs. Let $S \subseteq V(G)$ be a subset of vertices that represents where the PMUs are placed, which will observe other vertices in the graph using the following rules. The initialization step (or domination step) is all vertices in $S$ as well as all neighbors of vertices in $S$ are observed. The propagation step (or zero forcing step) is every  vertex which is the only unobserved neighbor of some observed vertex becomes observed. The initialization step can only occur once while the propagation step can occur multiple times until no other vertices can be observed. If the initial set $S$ eventually observes all the other vertices in the graph, it is called a {\it power dominating set}. The power domination number of $G$, $\gamma_P(G)$, is the cardinality of the smallest power dominating set of $G$. Recall that a dominating set of a graph $G$ is a set $S \subseteq V(G)$ such that every vertex in $V(G) - S$ is adjacent to a vertex in $S$. The \emph{domination number} of $G$, denoted $\gamma(G)$, is the minimum cardinality among all dominating sets of $G$. Due to the initialization step, any dominating set is a power dominating set. Thus, $\gamma_P(G) \leq \gamma(G)$. In addition, repeatedly applying the propagation step without first applying the initialization step is a graph searching process known as zero forcing. In particular, a set $S \subseteq V(G)$ is a {\it zero forcing set} if all vertices in $G$ are observed using only the propagation step. Since any zero forcing set is a power dominating set, $\gamma_P(G) \leq Z(G)$, where $Z(G)$ is the zero forcing number of $G$, or the size of the smallest zero forcing set of $G$. Zero forcing was first introduced in \cite{AIM} and has been extensively studied, see \cite{ARSS-2020,BBFHHSDH-2010,BBFHHSDH-2013,KBDF,DH-2020,GR-2018}. 

While upper bounds on the power domination number have been extensively studied, not many lower bounds on the power domination number have been given. In \cite{KBDF}, Benson et al. provide the following lower bound, which relies on knowing the zero forcing number of the graph. Note that $\Delta(G)$ denotes the maximum degree of a vertex of $G$.

\begin{theorem} \cite{KBDF}
Let $G$ be a graph that has an edge. Then $\left \lceil \frac{Z(G)}{\Delta(G)} \right \rceil \leq \gamma_P(G)$, and this bound is tight.
\end{theorem}

We construct a lower bound for $\gamma_P(G)$ based on a partitioning of the vertices of $G$ into sets such that their complements are not power dominating sets of $G$. This is very similar to the study of zero blocking sets, which were defined in \cite{KSTV-2020} by Karst et al., where a zero blocking set is the complement of a set which is not a zero forcing set. Similarly, Fetcie et al. studied failed zero forcing numbers in a graph \cite{FJS15} and Glasser et al. studied the failed power domination number of a graph \cite{GJLR-2020}. The main objective of this paper is to provide a new lower bound on the power domination number of the Cartesian product of two graphs. Recall that for graphs $G$ and $H$, the Cartesian product $G\Box H$ has vertex set
$V(G \Box H) = \{(g,h)\,:\, g\in V(G), h \in V(H)\}$.  Two vertices $(g_1,h_1)$ and $(g_2,h_2)$ are adjacent in $G\Box H$ if either $g_1=g_2$ and $h_1h_2\in E(H)$ or $h_1=h_2$ and $g_1g_2 \in E(G)$. Koh and Soh provide the following bound on the power domination number of the Cartesian product of two graphs in \cite{pdprod}. 

\begin{lemma}\cite{pdprod}\label{thm:Cartlower} For any connected graphs $G$ and $H$, $\gamma_P(H) \leq \gamma_P(G\Box H)$.
\end{lemma}

It was also claimed in \cite{pdprod} that $\gamma_P(G\Box T) \ge \gamma_P(G)\gamma_P(T)$ for any graph $G$ and any tree $T$. In \cite{pdcubic} the authors pointed out a flaw in the proof provided in \cite{pdprod}. We are able to resolve the problem when both $G$ and $T$ are trees; that is, we show the following. 

\begin{theorem} For any trees $T_1$ and $T_2$, $\gamma_P(T_1)\gamma_P(T_2) \leq \gamma_P(T_1 \Box T_2)$.
\label{cartlowerboundtree}
\end{theorem}

More generally, we provide a new lower bound for $\gamma_P(G \Box H)$ based on both factors. In fact, note that for any graph invariant $\psi$, the inequality $\psi(G\Box H) \ge \psi(G)\psi(H)$ is referred to as a Vizing-like bound after Vizing's infamous conjecture \cite{Vizing} from 1968 that states $\gamma(G\Box H) \ge \gamma(G)\gamma(H)$. We are particularly interested in knowing whether $\gamma_P(G\Box H) \ge\gamma_P(G)\gamma_P(H)$ for any pair of graphs $G$ and $H$.  

This paper is organized as follows. We provide useful definitions and terminology used throughout the paper in Section~\ref{sec:defn}. In Section \ref{sec:trees}, we investigate power domination in trees. In Section \ref{sec:prod}, we construct new lower bounds for $\gamma_P(G)$ and $\gamma_P(G \Box H)$ for graphs $G$ and $H$. In addition, we show a Vizing-like bound for the Cartesian product of two trees holds for the power domination number and give an equivalent lower bound for $Z(G\Box H)$. In Section 4, we examine the interaction between power domination and cut-sets. 

\subsection{Definitions and Terminology}\label{sec:defn}
We consider only simple, finite graphs. Given a graph $G = (V(G), E(G))$ and a vertex $v \in V(G)$, the {\it open neighborhood of $v$}, denoted $N_G(v)$, is the set of all vertices adjacent to $v$ and the {\it closed neighborhood of $v$} is defined to be $N_G[v] = N_G(v) \cup \{v\}$. The {\it degree of $v$} is $\deg_G(v) = |N_G(v)|$. When the context is clear, we simply write $N(v), N[v]$, and $\deg(v)$. A vertex of degree $1$ is referred to as a {\it leaf}. The distance between two vertices $u$ and $v$ is denoted $d_G(u,v)$.  A set $S\subseteq V(G)$ is a power dominating set if all vertices of $G$ are eventually observed according to the following rules. 
\begin{itemize}
\item \textbf{Initialization Step (Domination Step):} All vertices in $S$ as well as all neighbors of vertices in $S$ are observed.
\item \textbf{Propagation Step (Zero Forcing Step):} Every vertex which is the only unobserved neighbor of some observed vertex becomes observed.
\end{itemize}
In this paper, we will consider applying the above observation rules to a set of vertices that is not necessarily a power dominating set of $G$. In particular, if $S\subseteq V(G)$ and we apply the domination step and then the zero forcing step repeatedly until no further vertices can be observed, we will say that $v \in V(G)$ is {\it power dominated by $S$} if $v$ is eventually observed in the above process, and we will say $v$ is {\it not power dominated by $S$} if it is not eventually observed.

\section{Trees}\label{sec:trees}
It is known that for some graphs $G$ and any spanning tree $T$ of $G$ that $Z(T) < Z(G)$. Similarly, for some graphs $G$ and any spanning tree $T$ of $G$, $\gamma_P(T)<\gamma_P(G)$. Consider the graph $G$ depicted in Figure~\ref{orderadj}. To see that $\gamma_P(G)=3$, consider $\Pi=\bigcup_{i=1}^3 \Pi_i$, with $\Pi_1=\{u_i\}_{i=1}^4$, $\Pi_2=\{v_i\}_{i=1}^4$, and $\Pi_3=\{w_i\}_{i=1}^6$. Note that $\Pi$ is a partition of the vertices of $G$ such that for each $i \in \{1,2,3\}$, $V(G)-\Pi_i$ is not a power dominating set of $G$.  Due to this, any power dominating set of $G$ must contain at least one vertex from each $\Pi_i$, and so $\gamma_P(G) \geq 3$.  On the other hand, $\{v_1,u_1,w_1\}$ is a power dominating set of $G$, and thus $\gamma_P(G)=3$.  However, each of the four spanning trees of $G$, given by deleting any one of the four edges $w_2w_3$, $w_3w_5$, $w_5w_6$, or $w_2w_6$, have spider cover number $2$, and thus for $T$ a spanning tree of $G$, $\gamma_P(T)=2$.

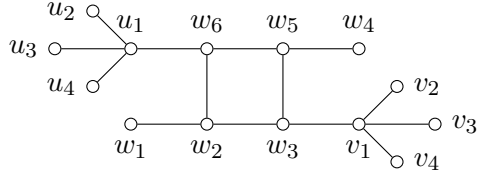
\begin{figure}[h]
\begin{center}
\begin{tikzpicture}[]
\tikzstyle{vertex}=[circle, draw, inner sep=0pt, minimum size=6pt]
\tikzset{vertexStyle/.append style={rectangle}}
	\vertex (1) at (1,1) [scale=.75, label=below:$w_1$] {};
	\vertex (2) at (2,1) [scale=.75, label=below:$w_2$] {};
	\vertex (3) at (3,1) [scale=.75, label=below:$w_3$] {};
	\vertex (4) at (4,1) [scale=.75, label=below:$v_1$] {};
	\vertex (5) at (5,1) [scale=.75, label=right:$v_3$] {};
	\vertex (6) at (4,2) [scale=.75, label=above:$w_4$] {};
	\vertex (7) at (3,2) [scale=.75, label=above:$w_5$] {};
	\vertex (8) at (2,2) [scale=.75, label=above:$w_6$] {};
	\vertex (9) at (1,2) [scale=.75, label=above:$u_1$] {};
	\vertex (10) at (0,2) [scale=.75, label=left:$u_3$] {};
	\vertex (11) at (4.5,0.5) [scale=.75, label=right:$v_4$] {};
	\vertex (12) at (4.5,1.5) [scale=.75, label=right:$v_2$] {};
	\vertex (13) at (0.5,1.5) [scale=.75, label=left:$u_4$] {};
	\vertex (14) at (0.5,2.5) [scale=.75, label=left:$u_2$] {};
	\path
	(1) edge (2)
	(2) edge (3)
	(3) edge (4)
	(4) edge (5)
	(6) edge (7)
	(7) edge (8)
	(8) edge (9)
	(9) edge (10)
	(11) edge (4)
	(12) edge (4)
	(13) edge (9)
	(14) edge (9)
	(7) edge (3)
	(8) edge (2);
\end{tikzpicture}
\end{center}
\caption{An example where no spanning tree has power domination number $\gamma_P(G)$}
\label{orderadj}
\end{figure}

Our original motivation for considering the power domination of spanning trees of a given graph was to perhaps lower bound $\gamma_P(G\Box H)$ with some value that involved $\min\{\gamma_P(T): \text{ $T$ is a spanning tree of $G$}\}$ or $\max\{\gamma_P(T): \text{$T$ is a spanning tree of $G$}\}$. Unfortunately, \[\gamma_P(G) -\min\{\gamma_P(T): \text{ $T$ is a spanning tree of $G$}\}\] could be very large. In addition, as witnessed by the example above, and the fact that $\gamma_P(P_3 \Box P_6)=1$ yet $P_3 \Box P_6$ has spanning trees with power domination number $2$, it follows that $\gamma_P(G)$ and $\max\{\gamma_P(T): \text{$T$ is a spanning tree of $G$}\}$ are incomparable. So instead we explored properties of trees that would be useful in the study of $\gamma_P(T_1\Box T_2)$ where $T_1$ and $T_2$ are both trees. In the following result, we will consider rooting a tree $T$ at some vertex $r$ and define $A_i = \{v \in V(T): d_T(v, r) = i\}$. If $w \in A_i$ for some $i \ge 1$ and $z$ is a neighbor of $w$ in $A_{i+1}$, we will refer to $z$ as a {\it descendant of $w$}. 

\begin{theorem}\label{thm:treepart} For any tree $T$ on at least two vertices, there exists a partition $\Pi_1 \cup \cdots \cup \Pi_{\spider(T)}$ such that $T[\Pi_i]$ is connected and for each $i \in [\spider(T)]$, $\Pi_i$ contains two distinct leaves of $T$, say  $w_i$ and $x_i$, such that the  $w_ix_i$-path $P$ in $T$ satisfies that if $t$ on $P$ is power dominated by $V(T) - \Pi_i$, then both neighbors of $t$ on $P$ are not power dominated by $V(T) - \Pi_i$. 
\label{thm:treeell}
\end{theorem}

\begin{proof} Throughout the proof, we will say use the following terminology. Suppose $\Pi_1 \cup  \cdots \cup \Pi_{\spider(T)}$ is a partition of $V(T)$ and for some $i\in [\spider(T)]$, $\Pi_i$ contains two distinct leaves of $T$, say $w_i$ and $x_i$, such that the  $w_ix_i$-path $P$ in $T$ satisfies that if $t$ on $P$ is power dominated by $V(T) - \Pi_i$, then both neighbors of $t$ on $P$ are not power dominated by $V(T) - \Pi_i$. In this case, we will say that $P$ satisfies Condition (1).  We proceed by induction on $m=\spider(T)$. If $\spider(T) = 1$, then our partition contains only one set, namely $\Pi_1 =V(T)$, in which case $V(T)$ contains at least two leaves, say $w$ and $x$, and since there are no vertices outside $\Pi_1$, the statement of the theorem is vacuously true. 

Suppose the statement of the theorem is true for all trees $T'$ with $\spider(T') \le m$. Furthermore, suppose there exists a counterexample $T$ with $\spider(T) = m+1$. Among all such counterexamples, choose one of minimum order. First, we claim that we may assume that $T$ does not contain a pair of adjacent vertices each having degree $2$. To see this, suppose $uv \in E(T)$ where $\deg_T(u) = \deg_T(v) = 2$. Let $T'$ be the tree obtained from $T$ by contracting $u$ and $v$ and let $z$ represent the contracted vertex in $T'$. It is clear that $\spider(T) = \spider(T')$. Moreover, since $|V(T')| < |V(T)|$, $T'$ is not a counterexample to the theorem. Thus, there exists a partition $\Pi = \Pi_1 \cup \cdots \cup \Pi_{\spider(T')}$ that satisfies the statement of the theorem. Without loss of generality, we may assume $z \in \Pi_1$. By assumption, $\Pi_1$ contains two leaves of $T'$, say $w_1$ and $x_1$ such that the $w_1x_1$-path $P$ in $T'$ satisfies that if $t$ is on $P$ and is power dominated by $V(T') - \Pi_1$, then both neighbors of $t$ on $P$ are not power dominated by $V(T') - \Pi_1$. Consider the partition $\Pi' = \Pi_1' \cup \Pi_2 \cup \cdots \cup \Pi_{\spider(T')}$ of $V(T)$ where $\Pi_1' = (\Pi_1 - \{z\}) \cup \{u,v\}$. We claim that $\Pi'$ satisfies the statement of the theorem with respect to $T$. Note first that $w_1$ and $x_1$ are still leaves in $T$. If the $w_1x_1$-path $P'$ in $T'$ does not contain $z$, then $P'$ is the $w_1x_1$-path in $T$ and Condition (1) is still satisfied. Thus, we shall assume $z$ is on $P'$. It follows that the $w_1x_1$-path in $T$ is obtained from $P'$ by replacing $z$ with $uv$. Moreover, $\deg_{T'}(z) = 2$ and both neighbors of $z$ are in $\Pi_1$. Thus, both neighbors of $u$ and $v$ are also in $\Pi_1'$ and Condition (1) is still satisfied. Furthermore, for each $2 \le i \le \spider(T')$, each $\Pi_i$ contains two distinct leaves, say $w_i$ and $x_i$, such that the $w_ix_i$-path $P'$ in $T'$ satisfies Condition (1). $w_i$ and $x_i$ will also be leaves in $T$ and the $w_ix_i$-path in $T$ will still be $P'$, which means the statement of the theorem is still satisfied with respect to $T$. However, this is a contradiction as no such partition of $V(T)$ is assumed to exist. Therefore, we may assume for the remainder of the proof that no pair of adjacent vertices in $T$ each have degree $2$. 

Pick a leaf $\ell$ in $T$ and root $T$ at $\ell$. Define $A_i = \{v\in V(T): d_T(\ell, v) = i\}$ and assume $d= \max_v\{d_T(v, \ell)\}$. Note that all vertices in $A_d$ are leaves in $T$. Suppose first that $A_{d-1}$ contains a strong support vertex $x$.

\medskip

\noindent{\underline{\textbf{Case 1a}}} \ \ Suppose that the neighbor of $x$ in $A_{d-2}$ is $s$ where $\deg_T(s) >3$.

\medskip

 Let $T'$ be the tree in $T- \{sx\}$ that contains $s$.  Note that $\spider(T') \le \spider (T)$. Suppose first that $\spider(T') = \spider(T)$. Since $T$ was assumed to have minimum order, we can find a partition $\theta=\theta_1\cup \cdots \cup \theta_{\spider(T')}$ of $V(T')$ that satisfies the statement of the theorem. Reindexing if necessary, we shall assume $s \in \theta_1$. Consider the partition $\Pi = \Pi_1 \cup \theta_2 \cup \cdots \cup \theta_{\spider(T')}$ of $V(T)$ where $\Pi_1 = \theta_1 \cup (V(T)- V(T'))$. One can easily verify that this partition satisfies the statement of the theorem, which is a contradiction. Therefore, we shall assume $\spider(T') < \spider(T)$. Note that in this case $\spider(T') = \spider(T)-1$. Again, we may assume there exists a partition $\theta = \theta_1 \cup \cdots \cup \theta_{\spider(T')}$ of $V(T')$ satisfying the statement of the theorem where $s \in \theta_1$. Consider the partition $\Pi=\Pi_1 \cup \theta_1 \cup \cdots \cup \theta_{\spider(T')}$ where $\Pi_1 = V(T) - V(T')$. It is clear that $\Pi_1$ contains two leaves, say $\ell_1$ and $\ell_2$, adjacent to $x$ as $x$ is a strong support vertex. Moreover, the only $\ell_1\ell_2$-path  in $T$ is $\ell_1x\ell_2$ and $V(T) - \Pi_1$ does not power dominate $\ell_1$ or $\ell_2$. It is also clear that each $\theta_i$ for $2 \le i \le \spider(T')$ contains two leaves, say $w_i$ and $x_i$, such that the $w_ix_i$-path $P$ in $T'$ (and $T$) satisfies Condition (1). As $w_i$ and $x_i$ are still leaves in $T$ and every vertex on the $w_ix_i$-path in $T'$ has the same closed neighborhood in $T'$ as $T$, the $w_ix_i$-path in $T$ still satisfies Condition (1). Thus, we only need to consider $\theta_1$.  By assumption, there exist two leaves $w_1$ and $x_1$ in $T'$ such that the $w_1x_1$-path $P$ in $T'$ satisfies Condition (1). If $s$ is not on $P$, then we are done. So we shall assume $s$ is on $P$ and $w_1$ and $x_1$ are leaf descendants of $s$. If $s$ is power dominated by $\theta_2 \cup \cdots \cup \theta_{\spider(T')}$, then by assumption the neighbors of $s$ are not power dominated by $V(T') - \theta_1$ and we are done. Therefore, we shall assume $s$ is not power dominated by $\theta_2 \cup \cdots \cup \theta_{\spider(T')}$. This implies that $N_T(s) - \{x\}$ is contained in $\theta_1$. Since $\deg_T(s) > 3$, $s$ has at least two leaf descendants, say $r_1$ and $r_2$ (not necessarily different from $w_1$ and $x_1$), and the only $r_1r_2$-path $P'$ in both $T$ and $T'$ must satisfy Condition (1) with respect to $T$. Thus, $\Pi$ is such a partition of $T$, which is a contradiction. Hence, this case cannot occur. 

\vskip5mm

\noindent{\underline{\textbf{Case 2a}}} \ \ Suppose that the neighbor of $x$ in $A_{d-2}$ is $s$ where $\deg_T(s) = 3$. 

\medskip 
Write $N_T(s) \cap A_{d-1} = \{z, x\}$. Suppose first that $\deg_T(z) \ge3$. Thus, $z$ is a strong support vertex of $T$. Let $T'$ be the component of $T- \{sx\}$ that contains $s$. Similar to Case 1a, we may assume $\spider(T') = \spider(T)-1$ and $T'$ contains a partition $\theta = \theta_1 \cup \cdots \cup \theta_{\spider(T')}$ that satisfies the statement of the theorem where $s \in \theta_1$. We claim that $\Pi = \Pi_1 \cup \theta_1 \cup \cdots \cup \theta_{\spider(T')}$ is a partition of $V(T)$ satisfying the statement of the theorem where $\Pi_1 = V(T) - V(T')$. As in Case 1a, all we need to show is that $\theta_1$ contains two leaves, say $w_1$ and $x_1$, such that the $w_1x_1$-path in $T$ satisfies Condition (1).  By assumption, $\theta_1$ contains two leaves, say $w_1$ and $x_1$, such that the $w_1x_1$-path $P$ in $T'$ satisfies Condition (1) with respect to $T'$. Moreover, we may assume $s$ is not power dominated by $\theta_2 \cup\cdots \cup \theta_{\spider(T')}$. Thus, $z \in \theta_1$ and all descendants of $z$ are in $\theta_1$. Since $z$ is a strong support vertex, it has two leaves, say $r_1$ and $r_2$, such that the only $r_1r_2$-path in $T$ is $r_1zr_2$. This path satisfies Condition (1) and it follows that $\Pi$ is a partition satisfying the statement of the theorem, which is a contradiction.

Therefore, we shall assume $\deg_T(z) \le 2$. Let $r$ be the neighbor of $s$ in $A_{d-3}$. Let $T'$ be the component of $T - \{rs\}$ containing $r$. Suppose first that $T'$ contains no vertex of degree $3$ or more. It follows that $T'$ is a path. In this case, $\spider(T) = 2$ and $\Pi_1 \cup \Pi_2$ where $\Pi_1 = N_T[x]\ - \{s\}$ and $\Pi_2 = V(T) - \Pi_1$ is a partition of $V(T)$ that satisfies the statement of the theorem. Therefore, we may assume $T'$ contains at least one vertex of degree $3$ or more. As in Case 1a, we may assume $\spider(T') = \spider(T)-1$ and there exists a partition $\theta = \theta_1 \cup \cdots \cup \theta_{\spider(T')}$ of $V(T')$ that satisfies the statement of the theorem where $r \in \theta_1$. We claim that $\Pi = \Pi_1 \cup \theta_1' \cup \theta_2 \cup \cdots \cup \theta_{\spider(T')}$ is a partition of $V(T)$ that satisfies the statement of the theorem where $\Pi_1 = N_T[x] - \{s\}$ and $\theta_1' = \theta_1 \cup N_T[z]$. As before, we need only check that $\theta_1'$ contains a path that satisfies Condition (1).  By assumption, $\theta_1$ contains two leaves of $T'$, say $w_1$ and $x_1$, such that the $w_1x_1$-path $P$ in $T'$ satisfies Condition (1). Suppose first that one of $w_1$ or $x_1$ is $r$. Without loss of generality, we may assume $w_1 = r$. It follows that $|N_T(r)| = 2$. If $z$ is a leaf of $T$, set $w_1' = z$. If $z$ is not a leaf, then $z$ has exactly one leaf, call it $y$,  and set $w_1' = y$. We claim the $w_1'x_1$-path $P'$ in $T$ satisfies Condition (1). Assume first that $P' = v_0v_1\dots v_krsz$ where $v_0=x_1$. By assumption, if $v_i$ is power dominated by $\theta_2 \cup \cdots \cup \theta_{\spider(T')}$, then $v_{i-1}$ and $v_{i+1}$ are not. Moreover, $r$ is not power dominated by $\theta_2 \cup \cdots \cup \theta_{\spider(T')}$. Thus, even though $s$ is power dominated by $\Pi_1$, $r$ and $z$ will never be observed by $\Pi_1 \cup \theta_2 \cup \cdots \cup \theta_{\spider(T')}$. Thus, Condition (1) holds for $P'$. A similar argument when $w_1' = y$ yields another contradiction. Therefore, we shall assume neither $w_1$ nor $x_1$ is $r$. If $r$ is power dominated by $\theta_2 \cup \cdots \cup \theta_{\spider(T')}$, then $P'$ still satisfies Condition (1) with respect to $\Pi$ in $T$. Therefore, we shall assume that $r$ is not power dominated by $\theta_2 \cup \cdots \cup \theta_{\spider(T')}$. Note that $s$ will be power dominated by $\Pi_1$, but $r$ will not be observed by $s$. Thus, $P'$ still satisfies Condition (1) with respect to $\Pi$ in $T$. It follows that $\Pi$ is indeed a partition satisfying the statement of the theorem, which is a contradiction. 
\vskip5mm

\noindent\underline{\textbf{Case 3a}} \ \ Suppose that the neighbor of $x$ in $A_{d-2}$ is $s$ where $\deg_T(s) \le 2$. 

\medskip
Note that we may assume $\deg_T(s) = 2$ for otherwise $T$ is a spider and $\spider(T) = 1 < m+1$, which is a contradiction. It follows that we can write $N_T(s) = \{x, r\}$ where $r \in A_{d-3}$. Either $\deg_T(r) \ge 3$ or $r$ is a leaf of $T$ as $T$ is assumed to not contain two adjacent vertices of degree $2$. If $r$ is a leaf, then $r = \ell$ and $T$ is a spider. However, this contradicts the assumption that $\spider(T) >1$. Thus, we shall assume $\deg_T(r) \ge3$. If $\deg_T(r)>3$, then using similar arguments as in Case 1a where we replace $s$ with $r$ will lead to a contradiction. On the other hand, if $\deg_T(r) = 3$, then using similar arguments as in Case 2a where we replace $s$ with $r$ will lead to another contradiction. Hence, this case cannot occur either. 
\medskip

Having considered all cases where $A_{d-1}$ contains a strong support vertex,  we shall assume for the remainder of the proof that $A_{d-1}$ contains no strong support vertex. Thus, $A_{d-2}$ must contain a vertex of degree $3$ or more since we assumed $T$ does not contain two adjacent vertices of degree $2$. Choose $x$ in $A_{d-2}$ of degree $3$ or more that contains a leaf descendant in $A_d$.  Let $s$ be the neighbor of $x$ in $A_{d-3}$. 
\vskip5mm
\noindent\underline{\textbf{Case 1b}} Suppose first that $\deg_T(s) >3$.

\medskip
As in Case 1a, we let $T'$ be the component of $T- \{xs\}$ containing $s$. Using similar arguments to that in Case 1a, we may assume $\spider(T') = \spider(T) -1$ and there exists a partition $\theta = \theta_1 \cup \cdots \cup \theta_{\spider(T')}$ of $V(T')$ that satisfies the statement of the theorem where $s \in \theta_1$. Root $T'$ at $\ell$ and let $A_i' = A_i \cap V(T')$. As in Case 1a, consider the partition $\Pi = \Pi_1 \cup \theta_1 \cup \cdots \cup \theta_{\spider(T')}$ where $\Pi_1 = V(T) - V(T')$. As above, we need only to check that $\theta_1$ contains a path that satisfies Condition (1).  By assumption, there exist two leaves $w_1$ and $x_1$ in $\theta_1$ such that the $w_1x_1$-path $P$ in $T'$ satisfies Condition (1) with respect to $T'$.  If $s$ is power dominated by $\theta_2 \cup \cdots \cup \theta_{\spider(T')}$, then $P$ still satisfies Condition (1) with respect to $\Pi$. Therefore, we shall assume $s$ is not power dominated by $\theta_2 \cup \cdots \cup \theta_{\spider(T')}$. Thus, $N_T(s) - \{x\} \subseteq \theta_1$. Let $z_1, \dots, z_k$ represent the neighbors of $s$ in $A_{d-2}'$. We claim that all descendants of $z_i$ are in $\theta_1$. To see this, note that if $v$ is a neighbor of $z_i$ in $A_{d-1}'$, then $\deg_T(v) \le 2$ by assumption. Thus, $v$ must be in $\theta_1$ for otherwise $v \in \theta_r$ for $r\ne 1$ and $\theta_r$ does not contain a path that satisfies Condition (1). It follows that all descendants of $z_i$ are in $\theta_1$. If $\deg_T(z_i) \ge 3$, then $z_i$ has at least two leaf descendants, say $r_1$ and $r_2$, and the $r_1r_2$-path in $T$ satisfies Condition (1). On the other hand, if $\deg_T(z_i) \le 2$ for $i \in [k]$, then we can choose a leaf descendant of $z_i$ (or $z_i$ itself if $z_i$ is a leaf), and call it $r_i$ for $i \in [2]$. The $r_1r_2$-path in $T$ satisfies Condition (1). In either case, $\Pi$ is indeed a partition of $V(T)$ satisfying the statement of the theorem, which is a contradiction. Hence, this case cannot occur.
\vskip5mm

\noindent\underline{\textbf{Case 2b}} \ \ Next, assume $\deg_T(s) = 3$. Using arguments similar to that in Case 2a, we arrive at a similar contradiction.

\medskip

\noindent\underline{\textbf{Case 3b}} \ \ Finally, we shall assume $\deg_T(s) = 2$ and write $N_T(s) = \{x, r\}$. 
\medskip

It follows that $\deg_T(r) = 1$ or $\deg_T(r) \ge 3$. Note that if $r$ is a leaf, then $T$ is a spider which contradicts the assumption that $\spider(T)>1$. Therefore, we shall assume $\deg_T(r) \ge 3$. Let $T'$ be the component of $T- \{rs\}$ that contains $r$. Thus, $\spider(T') \le \spider(T)$. Using similar arguments as those in Case 1a, we may assume $\spider(T') = \spider(T)-1$ and there exists a partition $\theta = \theta_1 \cup \cdots \cup \theta_{\spider(T')}$ of $V(T')$ that satisfies the statement of the theorem where $r \in \theta_1$. As before, we root $T'$ at $\ell$ and let $A_i' = A_i \cap V(T')$. We first assume that $\deg_T(r) >3$ or some descendant of $r$ in $T'$ has degree $3$ or more. Consider $\Pi  = \Pi_1 \cup \theta_1\cup \cdots \cup \theta_{\spider(T')}$ where $\Pi_1 = V(T) - V(T')$. Let $z_1, \dots, z_k$ represent the descendants of $r$ in $T'$ in $A'_{d-3}$. Furthermore, let $y$ be the neighbor of $r$ in $A'_{d-5}$. Reindexing if necessary, we may assume for $1 \le j \le k$ that $\deg_T(z_i) \le 2$ for $i \in [j]$ and $\deg_T(z_i) \ge 3$ for $j+1 \le i \le k$. For each $i \in \{j+1, \dots, k\}$,  let $v_{i,1}, \dots, v_{i,\ell_i}$ be the neighbors of $z_i$ in $A'_{d-2}$. First note, since each $v_{i,j} \in A'_{d-2}$ and $A_{d-1}$ contains no strong support vertices, for each descendent $w$ of $v_{i,j}$, $\deg_T(w)\le 2$.  As a result, for each $v_{i,j}$, every descendent of $v_{i,j}$ is in the same element of $\Pi$ as $v_{i,j}$. Next note that if for some $i \in \{j+1, \dots, k\}$ that $v_{i,m}$ and $v_{i,n}$ are both in $\theta_1$, then we can find a leaf descendant of $v_{i,m}$ (possibly itself) and a leaf descendant of $v_{i,n}$ (possibly itself) such that the path between them satisfies Condition (1)  and $\Pi$ is indeed a partition of $V(T)$ that satisfies the statement of the theorem. Therefore, we shall assume for each $i \in \{j+1, \dots, k\}$ that at most one vertex of the form $v_{i,m}$ is in $\theta_1$. Now, by assumption, $\theta_1$ contains two leaves of $T'$, call them $w_1$ and $x_1$, such that the $w_1x_1$-path $P$ in $T'$ satisfies Condition (1). Note that if $r$ is power dominated by $\theta_2 \cup \cdots \cup \theta_{\spider(T')}$, then $P$ still satisfies Condition (1) with respect to $\Pi$. Therefore, we shall assume $r$ is not power dominated by $\theta_2 \cup \cdots \cup \theta_{\spider(T')}$. It follows that each $z_i$ is in $\theta_1$ and each $z_i$ where $\deg_T(z_i) \ge 2$ must have a neighbor in $\theta_1 \cap A'_{d-2}$ for otherwise $r$ is power dominated by $\theta_2 \cup \cdots \cup \theta_{\spider(T')}$. 

Next, we claim that we may assume $j < 2$. To see this, suppose both $z_1$ and $z_2$ have degree at most $2$ in $T$. Then we can find a leaf descendant of $z_i$ (possibly itself), call it $r_i$, for $i \in [2]$ and the $r_1r_2$-path in $T$ satisfies Condition (1). Therefore, we shall assume $j < 2$.

Suppose that $\deg_T(z_1) =1$ for the time being. Let $T''$ be the component of $T- \{ry\}$ containing $y$. For each $z_i$ where $2 \le i \le k$, we have assumed $z_i$ has exactly one neighbor in $\theta_1 \cap A'_{d-2}$. We shall assume $v_{i, \ell_i}$ is that neighbor. This implies that $\deg_T(v_{i,j}) \ge 3$ for $1 \le j \le \ell_i -1$ and $2 \le i \le k$. Furthermore, $\deg_T(v_{i, \ell_i}) \le 2$ for otherwise $v_{i, \ell_i}$ has two leaf descendants and the path between them satisfies Condition (1). For each $i \in \{2, \dots, k\}$, let $R_i$ represent $v_{i, \ell_i}$ together with any descendant of $v_{i, \ell_i}$. Reindexing if necessary, we may write 
\[\theta = \theta_1 \cup \theta_2 \cup \cdots \cup \theta_{\beta} \cup \theta_{\beta+1} \cup \cdots \cup \theta_{\spider(T')}\]
such that $\theta_{\alpha}$ for $2 \le \alpha \le \beta$ contains a vertex of $\cup_{i=2}^k \{v_{i,1}, \dots, v_{i,\ell_i-1}\}$. Thus, 
\[\beta-1 = \sum_{i=2}^k (\deg_T(z_i) - 2).\]
We claim that $\spider(T'') \ge \spider(T) - \beta$. To see this, pick a spider partition of $T''$, call it $\Sigma$ and define 
\[\Gamma_1 = (V(T) - V(T')) \cup \{r, z_1\}\]
\[\Gamma_{\alpha} = \begin{cases} 
\theta_{\alpha} & \text{if $\theta_{\alpha}$ contains $v_{i,j}$ for $2 \le i \le k, 1\le j \le \ell_i-2$}\\
\theta_{\alpha} \cup \{z_i\} \cup R_{i} & \text{if $\theta_{\alpha}$ contains $v_{i, \ell_i-1}$}
\end{cases}\]
Thus, $\Sigma \cup \Gamma_1 \cup \cdots \cup\Gamma_{\beta}$ is a spider partition of $T$ of cardinality $\spider(T'') + \beta$. It follows that $\spider(T'') \ge \spider(T) - \beta$. By the induction hypothesis and choice of minimal counterexample with respect to order,  we can find a partition $\Sigma = \Sigma_1 \cup \cdots \cup \Sigma_{\spider(T'')}$ of $V(T'')$ that satisfies the statement of the theorem where $y \in \Sigma_{\spider(T'')}$. Suppose first that either $y$ is a leaf of $T''$ such that there exists a leaf $w \in \Sigma_{\spider(T'')}$ where the $yw$-path $P$ in $T''$ satisfies Condition (1), or $y$ is not power dominated by $V(T'') - \Sigma_{\spider(T'')}$. We claim 
\[\Pi = \Sigma_1 \cup \cdots \cup \Sigma_{\spider(T'') -1} \cup \Sigma_{\spider(T'')}' \cup \Gamma_2 \cup \cdots \cup \Gamma_{\beta} \cup \Pi_1\] is a partition of $V(T)$ that satisfies the statement of the theorem where $\Sigma_{\spider(T'')}' = \Sigma_{\spider(T'')} \cup \{r, z_1\}$ and $\Gamma_{\alpha}$ is defined as above. We need only to check that $\Sigma'_{\spider(T'')}$ contains a path that satisfies Condition (1). We shall assume first that $y$ is a leaf of $T''$ such that there exists $w \in \Sigma_{\spider(T'')}$ where the $yw$-path $P$ in $T''$ satisfies Condition (1). We claim that the $z_1w$-path in $T$ satisfies Condition (1). Note that $\Sigma_1 \cup \cdots \cup \Sigma_{\spider(T'')-1}$ does not power dominate $y$. Moreover, although $\Pi_1$ power dominates $r$, $z_1$ and $y$ will never be observed with respect to $\Pi - \Sigma'_{\spider(T'')}$. Thus, $\Pi$ is in fact a partition of $V(T)$ satisfying the statement of the theorem. Therefore, we shall assume that $y$ is not a leaf of $T''$ and $y$ is not power dominated by $V(T'') - \Sigma_{\spider(T'')}$. By assumption, $\Sigma_{\spider(T'')}$ contains two leaves in $T''$, say $w_1$ and $x_1$ (neither of which is $y$), such that the $w_1x_1$-path $P$ in $T''$ satisfies Condition (1). Note that since $y$ is not power dominated by $V(T) - \Sigma'_{\spider(T'')}$, then $P$ still satisfies Condition (1) with respect to $T$. Thus, $\Pi$ is a partition of $V(T)$ satisfying the statement of the theorem. Therefore, we shall assume that $y$ is power dominated by $V(T'') - \Sigma_{\spider(T'')}$. In this case, consider 
\[\Pi' = \Sigma_1 \cup \cdots  \cup \Sigma_{\spider(T'')}\cup \Gamma_2 \cup \cdots \cup \Gamma_{\beta} \cup \Pi'_1\] 
where $\Pi_1' = \Pi_1 \cup \{r, z_1\}$. One can easily verify that $\Pi'$ is a partition of $V(T)$ that satisfies the statement of the theorem. In each case, we get a contradiction. Therefore, this case cannot occur. Furthermore, a similar argument can be used to show that $\deg_T(z_1) \ne 2$.

Next, we assume $\deg_T(z_i) \ge 3$ for all $i \in [k]$. Consider $T''$ obtained from $T$ where we remove all edges incident to $r$ other than $rs$ and $ry$. We claim that $\spider(T'') \ge \spider(T) - (\beta -1)$. As above, we may assume for each $z_i$ for $i \in [k]$ that $z_i$ has exactly one neighbor, namely $v_{i, \ell_i}$, in $\theta_1 \cap A'_{d-2}$. Reindexing if necessary, we may write 
\[\theta = \theta_1 \cup \theta_2 \cup \cdots \cup \theta_{\beta} \cup \theta_{\beta+1} \cup \cdots \cup \theta_{\spider(T')}\]
such that $\theta_{\alpha}$ for $2 \le \alpha \le \beta$ contains a vertex of $\cup_{i=1}^k \{v_{i,1}, \dots, v_{i,\ell_i-1}\}$. Thus, 
\[\beta-1 = \sum_{i=1}^k (\deg_T(z_i) - 2).\]
Pick a spider partition of $T''$, call it $\Sigma$, and define 
\[\Gamma_{\alpha} = \begin{cases} 
\theta_{\alpha} & \text{if $\theta_{\alpha}$ contains $v_{i,j}$ for $1 \le i \le k, 1\le j \le \ell_i-2$}\\
\theta_{\alpha} \cup \{z_i\} \cup R_{i} & \text{if $\theta_{\alpha}$ contains $v_{i, \ell_i-1}$}
\end{cases}\]
Thus, $\Sigma \cup \Gamma_2 \cup \cdots \cup \Gamma_{\beta}$ is a spider partition of $T$ of cardinality $\spider(T'') + (\beta-1)$. It follows that $\spider(T'') \ge \spider(T) - (\beta-1)$. By the induction hypothesis and choice of minimal counterexample with respect to order, we can find a partition $\Sigma = \Sigma_1 \cup \cdots \cup \Sigma_{\spider(T'')}$ of $V(T'')$ that satisfies the statement of the theorem where $r \in \Sigma_1$. We claim that 
\[\Pi = \Sigma_1\cup \cdots \cup \Sigma_{\spider(T'')} \cup \Gamma_2 \cup \cdots \cup \Gamma_{\beta}\]
where $\Gamma_{\alpha}$ is as defined above is a partition of $V(T)$ that satisfies the statement of the theorem. Note that we need only check that $\Sigma_1$ contains a path that satisfies Condition (1). By assumption, $\Sigma_1$ contains two leaves of $T''$, say $w_1$ and $x_1$ (neither of which is $r$), such that the $w_1x_1$-path $P$ in $T''$ satisfies Condition (1). If $r$ is not on $P$, then $\Pi$ satisfies the statement of the theorem. Therefore, we shall assume that $r$ is on $P$. It follows that one of $w_1$ or $x_1$ is a leaf descendant of $r$ in $T''$. In particular, we may assume $w_1$ is a leaf descendant of $x$ in $A_d$. Moreover, as $A_{d-1}$ contains no strong support vertex, all descendants of $r$ in $T''$ are in $\Sigma_1$. However, we have assumed $\deg_T(x) \ge 3$ and therefore, $\Sigma_1$ contains another leaf descendant of $x$, call it $t$, such that the $w_1t$-path in $T$ satisfies Condition (1). It follows that $\Pi$ satisfies the statement of the theorem.

Finally, consider when $\deg_T(r) = 3$ and no descendant of $r$ in $T'$ has degree $3$ or more. Write $N_T(r) = \{y, s, z\}$. In this case, we let $T''$ be the component of $T - \{ry\}$ containing $y$. As before, we may assume $\spider(T'') = \spider(T) -1$ and there exists a partition of $V(T'')$, say $\theta = \theta_1 \cup \cdots \cup \theta_{\spider(T'')}$ that satisfies the statement of the theorem where $y \in \theta_1$. If either $y$ is a leaf of $T''$ such that there exists a leaf in $T''$ $w \in \theta_1$ where the $yw$-path $P$ in $T''$ satisfies Condition (1), or $y$ is power dominated by $V(T'') - \theta_1$, then using similar arguments as above, one can easily verify that 
$\Pi = \Pi_1 \cup \theta'_1 \cup \cdots \cup \theta_{\spider(T'')}$ where $\Pi_1$ contains $s$ and all descendants of $s$ and $\theta_1'$ contains $\theta_1$ together with $r, z$, and any descendant of $z$  indeed satisfies the statement of the theorem. On the other hand, if $y$ is not a leaf of $T''$ and is not power dominated by $V(T'') - \theta_1$, then one can easily verify that $\Pi = \Pi_1 \cup \theta_1 \cup \cdots \cup \theta_{\spider(T'')}$ where $\Pi_1 = V(T) - V(T'')$ satisfies the statement of the theorem. 

\end{proof}

In the next section, we use Theorem~\ref{thm:treepart} to show that $\gamma_P(T_1\Box T_2) \ge \gamma_P(T_1)\gamma_P(T_2)$ for any trees $T_1$ and $T_2$. 

\section{Cartesian products} \label{sec:prod}
There are two main goals of this section. The first is to provide a lower bound for $\gamma_P(G\Box H)$. The second is to prove a  Vizing-like bound holds for the product of two trees. Given any graph $G$ and any partition $\Pi = \Pi_1 \cup \cdots \cup \Pi_k$ of $V(G)$ where $k\ge2$, we say $\Pi$ is a failed power dominating partition of $G$ if for each $i\in [k]$, $V(G) - \Pi_i$ does not power dominate $G$. If such a partition exists, we let $\ell_G= \max\{j: \Pi_1 \cup \cdots \cup \Pi_j \text{ is a failed power dominating partition of $G$}\}$. If no such partition exists, we let $\ell_G = 1$. Since $V(G) - \Pi_i$ does not power dominate $G$, any power dominating set of $G$ must contain at least one vertex from each $\Pi_i$. Thus, $\ell_G \leq \gamma_P(G)$. Figure \ref{fig:sitree} gives an example of a graph where $\ell_G = \gamma_P(G) = 4$ since $\Pi_1 \cup \Pi_2\cup \Pi_3\cup \Pi_4$ where $\Pi_1 = \{\ell_1, \ell_2, \ell_3, v_1\}$, $\Pi_2 = \{u_1, u_2, z, u_3, u_4, u_5, u_6\}$, $\Pi_3 = \{v_2, \ell_4, \ell_5, \ell_6\}$, and $\Pi_4 = \{v_3, \ell_7, \ell_8, \ell_9\}$ is a failed power dominating partition of the graph. In addition, $G = K_{3, 3}$ is an example of a graph where $\ell_G = 1 < \gamma_P(G) = 2$ since no failed power dominating partition of $V(K_{3, 3})$ exists.

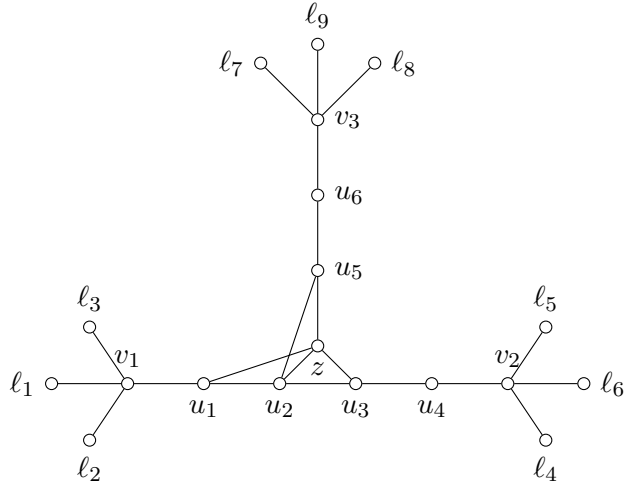
\begin{figure}[h]
\begin{center}
\begin{tikzpicture}[]
\tikzstyle{vertex}=[circle, draw, inner sep=0pt, minimum size=6pt]
\tikzset{vertexStyle/.append style={rectangle}}
	\vertex (1) at (0,0) [scale=.75, label=left:$\ell_1$] {};
	\vertex (2) at (.5,-.75) [ scale=.75, label=below:$\ell_2$] {};
	\vertex (3) at (1, 0) [ scale=.75, label=above:$v_1$] {};
	\vertex (4) at (.5, .75) [scale=.75, label=above:$\ell_3$] {};
	\vertex (5) at (2, 0) [scale=.75, label=below:$u_1$] {};
	\vertex (6) at (3, 0) [scale=.75, label=below:$u_2$] {};
	\vertex (7) at (3.5,.5) [scale=.75, label=below:$z$] {};
	\vertex (8) at (4,0) [scale=.75, label=below:$u_3$] {};
	\vertex (9) at (5,0) [scale=.75, label=below:$u_4$] {};
	\vertex (10) at (6.5,-.75) [scale=.75, label=below:$\ell_4$] {};
	\vertex (11) at (6,0) [scale=.75, label=above:$v_2$] {};
	\vertex (12) at (6.5,.75) [scale=.75,label=above:$\ell_5$] {};
	\vertex (13) at (7,0) [scale=.75, label=right:$\ell_6$] {};
	\vertex (14) at (3.5,1.5) [scale=.75, label=right:$u_5$] {};
	\vertex (15) at (3.5,2.5) [scale=.75, label=right:$u_6$] {};
	\vertex (16) at (3.5,3.5) [scale=.75, label=right:$v_3$] {};
	\vertex (17) at (2.75,4.25) [scale=.75, label=left:$\ell_7$] {};
	\vertex (18) at (4.25,4.25) [scale=.75, label=right:$\ell_8$] {};
	\vertex (19) at (3.5,4.5) [scale=.75, label=above:$\ell_9$] {};

	\path
	(1) edge (3)
	(2) edge (3)
	(4) edge (3)
	(3) edge (5)
	(5) edge (6)
	(6) edge (7)
	(7) edge (8)
	(8) edge (9)
	(9) edge (11)
	(10) edge (11)
	(12) edge (11)
	(13) edge (11)
	(7) edge (14)
	(14) edge (15)
	(15) edge (16)
	(17) edge (16)
	(18) edge (16)
	(19) edge (16)
	(6) edge (14)
	(6) edge (8)
	(5) edge (7)

	;
\end{tikzpicture}
\end{center}
\caption{An example where $\ell_G  =\gamma_P(G)$}
\label{fig:sitree}
\end{figure}

We can use failed power dominating partitions to provide a lower bound for the power domination number of the Cartesian product of two graphs. 

\begin{theorem} For any graphs $G$ and $H$, $\ell_G \ell_H \leq \gamma_P(G \Box H)$. If $\ell_G = \gamma_P(G)$ and $\ell_H = \gamma_P(H)$, then $\gamma_P(G)\gamma_P(H) \leq \gamma_P(G \Box H)$.
\label{cartlowerbound}
\end{theorem}
\begin{proof}
We may assume there exists a failed power dominating partition $\Pi = \Pi_1 \cup \cdots \cup \Pi_{\ell_G}$ of $G$ and failed power dominating partition $\Sigma = \Sigma_1 \cup \cdots \cup \Sigma_{\ell_H}$ of $H$. We claim that $\{\Pi_i\times \Sigma_j: 1 \le i \le \ell_G, 1 \le j \le \ell_H\}$ is a failed power dominating partition of $G\Box H$. To see this, note that since $\Pi$ is a failed power dominating partition of $G$, for each $i \in [\ell_G]$ there exists $U_i \subseteq \Pi_i$ such that  for each $u \in U_i$, $u$ is not power dominated by $V(G) - \Pi_i$. For each $i\in [\ell_G]$ we let $U_i$ represent the set of all vertices in $\Pi_i$ that are not power dominated by $V(G) - \Pi_i$. Similarly, for each $j \in [\ell_H]$, we let $V_j$ represent the set of all vertices in $\Sigma_j$ that are not power dominated by $V(H) - \Sigma_j$. We claim that for fixed $i \in [\ell_G]$ and $j \in [\ell_H]$ that no vertex of $U_i \times V_j$ will be power dominated by $V(G\Box H) - (\Pi_i \times \Sigma_j)$. Suppose to the contrary that $k$ is the first timestep where vertices from $U_i \times V_j$ are observed and $(u, v) \in U_i \times V_j$ is one vertex observed at timestep $k$. It follows that either $u$ is not an isolate in $G$ or $v$ is not an isolate in $H$ (or both). Moreover,  $(u,v)$ is observed by some vertex $(w, z)$ where either $w=u$ or $z=v$. We shall assume first that $w=u$. It follows that $z$ is power dominated by $V(H) - \Sigma_j$ as $z\not\in V_j$. Note that some vertex of $N_H(z) \cap V_j$ other than $v$, call it $s$, is not power dominated by $V(H) - \Sigma_j$, otherwise $z$ will eventually observe $v$. Thus, $(u,z)$ has two unobserved neighbors, namely $(u, v)$ and $(u,s)$, and therefore $(u,z)$ cannot observe $(u,v)$. A similar argument can be used if instead we had assumed $z=v$. Thus, no vertex of $U_i \times V_j$ is power dominated by $V(G\Box H) - (\Pi_i \times \Sigma_j)$ and $\{\Pi_i\times \Sigma_j: 1 \le i \le \ell_G, 1 \le j \le \ell_H\}$ is a failed power dominating partition of $G\Box H$. 
\end{proof}

Haynes et al. proved in \cite{electricgrids} that for any tree $T$, $\gamma_P(T) = \spider(T)$. Using this fact and Theorems \ref{thm:treeell} and \ref{cartlowerbound}, we can now provide a lower bound of $\gamma_P(T_1 \Box T_2)$ based on $\gamma_P(T_1)$ and $\gamma_P(T_2)$.

\begin{corollary} For any tree $T$, $\ell_T = \spider(T)$. 
\end{corollary}
\begin{proof} By Theorem~\ref{thm:treeell}, there exists a partition $\Pi = \Pi_1 \cup \cdots \cup \Pi_{\spider(T)}$ of $V(T)$ such that for each $i \in [\spider(T)]$, $\Pi_i$ contains two leaves of $T$, say $w_i$ and $x_i$, such that the $w_ix_i$-path $P$ in $T$ satisfies that if $t$ on $P$ is power dominated by $V(T) - \Pi_i$, then both neighbors of $t$ on $P$ are not power dominated by $V(T) - \Pi_i$. Therefore, $w_i$ and $x_i$ will not be power dominated by $V(T) - \Pi_i$ and $\Pi$ is indeed a failed power dominating partition of $T$. 
\end{proof}

Combining Theorem~\ref{cartlowerbound} and the above corollary, we have now proven the following.  

\medskip
\noindent \textbf{Theorem~\ref{cartlowerboundtree}} \emph{For any trees $T_1$ and $T_2$, $\gamma_P(T_1)\gamma_P(T_2) \leq \gamma_P(T_1 \Box T_2)$.} 
\medskip

We also point out that there are other graphs where $\ell_G= \gamma_P(G)$. To see this, we first provide a lower bound on $\ell_G$ based on a specific type of partition.

\begin{observation}\label{obs:sp} Let $G$ be a graph that has a partition $\Pi_1 \cup \cdots \cup \Pi_{k}$ with $k\ge2$ such that for each $i \in[k]$, there exists $U_i \subseteq \Pi_i$ such that $N_G[U_i] \subseteq \Pi_i$ and for each $x \in \Pi_i-U_i$ where $N(x) \cap U_i \ne \emptyset$, $x$ has at least two neighbors in $U_i$. Then $k \le \ell_G$. 
\end{observation}


The following are some examples of graphs for which there exists a partition that satisfies the statement of Observation~\ref{obs:sp} and $k=\ell_G$. Zhao et al. classified all claw-free cubic graphs of order $n$ whose power domination number equals $n/4$ in \cite{ZKC_cubic}. Recall that a diamond is $K_4-e$ and a necklace of diamonds is constructed by taking $k$ disjoint diamonds and adding a matching between the vertices of degree $2$ so that the resulting graph is connected and cubic. Note that if $G$ is a necklace of diamonds, then $\gamma_P(G) = \gamma(G)$. Moreover, if $G$ is a necklace of diamonds and we define $\Pi = \Pi_1 \cup \cdots \cup \Pi_{n/4}$ where each $\Pi_i$ induces a diamond, then $\Pi$ satisfies the statement of Observation~\ref{obs:sp}. Thus, $\ell_G = \gamma_P(G)$. In that same paper \cite{ZKC_cubic}, all connected graphs of order $n$ with power domination number equal to $n/3$ were classified. They constructed a family of graphs $\mathcal F$ in the following way. Let $H$ be any connected graph and write $V(H) = \{u_1, \dots, u_n\}$. Let $G$ be the graph obtained from $H$ by letting $V(G)$ $= V(H) \cup \{u_{i1}, u_{i2}: i \in [n]\}$ and $E(G)$ $= E(H) \cup \{u_iu_{i1}, u_iu_{i2}: i \in [n]\}$ and optionally adding the edge $u_{i1}u_{i2}$ for each $i \in [n]$.

\begin{theorem} \cite{ZKC_cubic} If $G$ is a connected graph of order $n \geq 3$, then $\gamma(G) \leq n/3$ with equality if and only if $G \in \mathcal{F} \cup \{K_{3, 3}\}$.
\end{theorem}

\begin{corollary} For each $G \in \mathcal{F}$, $\ell_G = \gamma_P(G)$. 
\end{corollary}
\begin{proof} Let $G$ $\in \mathcal{F}$ be constructed from $H$ with $V(H) = \{u_1, \dots, u_n\}$ as described above. Let $\Pi = \Pi_1 \cup \cdots \cup \Pi_{\gamma_P(H')}$ where $\Pi_i = \{u_i, u_{i1}, u_{i2}\}$ for each $i \in [n]$. Note that $V(G)- \Pi_i$ does not power dominate $u_{i1}$ or $u_{i2}$ for each $i \in [n]$. Thus, $\Pi$ is a failed power dominating partition. 
\end{proof}

  \subsection{Connections to zero forcing} \label{sec:zf}

As noted earlier, there is significant interaction between power dominating sets and zero forcing sets of a graph. In fact, many of the results of this  section can be translated into results about the zero forcing number. Given any graph $G$ and any partition $\Pi = \Pi_1 \cup \cdots \cup \Pi_k$ of $V(G)$ where $k\ge2$, we say $\Pi$ is a failed zero forcing partition of $G$ if for each $i\in [k]$, $V(G) - \Pi_i$ does not zero force $G$. If such a partition exists, we let $z_G= \max\{j: \Pi_1 \cup \cdots \cup \Pi_j \text{ is a failed zero forcing partition of $G$}\}$. If no such partition exists, we let $z_G = 1$. If $V(G) - \Pi_i$ does not power dominate $G$, then $V(G) - \Pi_i$ does not zero force $G$. In addition, since $V(G) - \Pi_i$ does not zero force $G$, any zero forcing set of $G$ must contain at least one vertex from each $\Pi$. Thus, $\ell_G \leq z_G \leq Z(G)$. Using similar techniques to those used in Theorem \ref{cartlowerbound}, we can provide a lower bound for the zero forcing number of the Cartesian product of two graphs.

\begin{theorem} For any graphs $G$ and $H$, $\ell_G \ell_H \leq z_Gz_H \leq Z(G \Box H)$. If $z_G = Z(G)$ and $z_H = Z(H)$, then $Z(G)Z(H) \leq Z(G \Box H)$.
\end{theorem}

\section{Power Domination and Cut-sets}
In this section we introduce bounds which can be used when trying to determine the power domination number for graphs with cut-sets (or vertex-cuts) of arbitrarily large size.  Specifically, for a given cut-set we consider the power domination number of the components (as well as the power domination number of the closed neighborhoods of the components) yielded by removing the cut-set from the graph to establish both an upper bound and a lower bound for the power domination number of the entire graph in terms of the power domination numbers of these components.  We then show that the bounds are tight by identifying classes of graphs which attain these maximum and minimum values.  The graphs providing witness to the tightness of the bounds are additional classes for which the failed power dominating partition number and power domination number are equal. We note that the lower bound in the following result is similar to that given in \cite{HS-thesis} in the context of zero forcing.

\begin{theorem}\label{thm:lowercut}
Let $G$ be a graph, $C$ a cut-set of $G$, and write $G-C = H_1 \cup \cdots \cup H_m$ where each $H_i$ is a component of $G-C$. Furthermore, let $K_i$ be the graph induced by $V(H_i) \cup C$.  Then
\[\sum_{i=1}^m\gamma_P(K_i)-(m-1)\abs{C} \leq \gamma_P(G) \leq \sum_{i=1}^m\gamma_P(H_i)+\abs{C}.\]
\end{theorem}

\begin{proof}
Note that if we let $B_i$ be a power dominating set for $H_i$ for $1 \le i \le m$, then $C \cup \bigcup_{i=1}^m B_i$ is a power dominating set of $G$. This establishes the upper bound. To show the lower bound holds, let $B$ be a power dominating set of $G$ and let $B_i = V(H_i)\cap B$ for each $1 \le i \le m$. Furthermore, we let $Y$ represent the vertices in $C-B$ that are observed by a vertex in $C\cap B$, and for each $i \in [m]$ we let $C_i$ represent the vertices in $C-B - Y$ that are observed by some vertex in $H_i$. For each $i\in [m]$, we claim that $B_i \cup (C- Y - C_i)$ is a power dominating set of $K_i$. To see this, we claim that if vertex $w$ in $V(K_i)$ is observed with respect to $B$ by time step $k$, then $w$ is observed in $K_i$ with respect to $B_i \cup (C - Y - C_i)$ by time step $k$. It is clear that this is the case for the initialization step for if $w \in N_G[B]$, then $w \in N_{K_i}[B_i \cup (C - Y - C_i)]$. So we shall assume for some $j \in \mathbb{N}$ that if $w \in V(K_i)$ is observed with respect to $B$ by some time step $j$, then $w$ is observed in $K_i$ with respect to $B_i \cup (C - Y - C_i)$ by time step $j$. Now let $y$ be some vertex in $ V(K_i)$ that is observed at time step $j+1$. Thus, in $G$ it is observed by some vertex $z$ at which point $N_G(z) - \{y\}$ consists of observed vertices. If $z \not\in V(K_i)$, then $y \in C$ and $y \in B_i \cup (C - Y - C_i)$. Thus, we may assume $z \in V(K_i)$. Each vertex of $N_G(z) - \{y\}$ is either (1) not in $K_i$ or (2) already observed by time step $j$. It follows that $y$ will still be observed by time step $j+1$. Thus, $B_i \cup (C - Y - C_i)$ is indeed a power dominating set of $K_i$. It follows that 
\begin{eqnarray*}
\sum_{i=1}^m\gamma_P(K_i) &\le& \sum_{i=1}^m|B_i \cup (C - Y - C_i)|\\
&=& \sum_{i=1}^m|B_i| + \sum_{i=1}^m|C \cap B| + \sum_{i=1}^m|C-B-Y - C_i|\\
&=& \gamma_P(G) + \sum_{i=1}^{m-1}|C\cap B| + \sum_{i=1}^m |C-B-Y - C_i|\\
&\le& \gamma_P(G) + (m-1)|C|,
\end{eqnarray*}
where the last inequality holds since we have counted each vertex of $C_i$ $m-1$ times in the previous summation. 

\end{proof}

Note that Theorem~\ref{thm:lowercut} can be stated more generally in that we could have instead assumed that $\{{\rm{comp}}(H_i)\}_{i=1}^m$ is a partition of ${\rm{comp}}(G-C)$ where ${\rm{comp}}(G)$ is used to denote the components of a graph $G$. Having established the bounds we now construct a class of graphs which witness the tightness of both bounds, and thus that the upper bound and lower bound are in fact sometimes equal. In fact, we will show that we can partition the vertices of these graphs into a failed power dominating partition such that $\ell_G = \gamma_P(G)$.

\begin{proposition}\label{constructcut1}
For each pair $m,s \in \mathbb N$ with $m \geq 2$ and $s \geq 1$ there is a graph $G_{m,s}$ such that there exists a cut-set $C$ of $G_{m,s}$ of size $s$ where $G_{m,s} - C =  \cup_{i=1}^m H_i$ with each $H_i$ a  component of $G_{m,s}-C$, $K_i$ is the graph induced by $V(H_i)\cup C$,  and 
\[\ell_{G_{m,s}} = \gamma_P(G_{m,s})=\sum_{i=1}^m \gamma_P(H_i)+\abs{C}=\sum_{i=1}^m \gamma_P(K_i)-(m-1)\abs{C}\]
\end{proposition}

\begin{proof}
Fix $m, s \in\mathbb{N}$ where $m \ge 2$ and $s \ge 1$. For each $i\in [m]$, let $H_i$ be the graph obtained from $K_{1, s+2}$ by subdividing each edge once. Let $v_i$ be the center of $H_i$ and let $x^i_1, \dots, x^i_{s+2}$ represent the leaves in $H_i$. We construct $G_{m,s}$ from the disjoint union of $\cup_{i=1}^mH_i$ and a set of vertices $\{d_j\}_{j=1}^s$ by adding each edge of the form $d_ju$ where $u = x^i_j$ for each $i \in [m]$ or $u$ is the support vertex of $x_j^i$ in $H_i$ for each $i\in [m]$. Note that $\gamma_P(G_{m, s}) \le m+s$ as $\{v_i\}_{i=1}^m \cup \{d_j\}_{j=1}^s$ is a power dominating set of $G_{m, s}$. To see that $m+s \le \gamma_P(G_{m, s})$, we construct a partition of $V(G_{m, s})$ that satisfies the statement of Observation~\ref{obs:sp}. For each $j \in [s]$, let $\Pi_j = N_{G_{m, s}}[d_j]$. For each $i \in \{1, \dots, m\}$, let $\Pi_{s+i} = \{v_i, x_{s+1}^i, x_{s+2}^i, y_{s+1}^i, y_{s+2}^i\}$ where $y_{s+\alpha}^i$ is the support vertex of $x_{s+\alpha}^i$ in $H_i$ for $\alpha\in[2]$. Certainly $\Pi = \Pi_1 \cup \cdots \cup \Pi_s \cup \Pi_{s+1}\cup \cdots \cup \Pi_{s+m}$ is a partition of $V(G_{m, s})$. For each $j \in [s]$, set $U_j = \{d_j\} \cup \{x_j^i: i \in [m]\}$. Note that every vertex in $\Pi_j - U_j$ is adjacent to two vertices in $U_j$. For each $i \in [m]$, set $U_{s+i} = \Pi_{s+i} - \{v_i\}$. Thus, $v_i$ is adjacent to exactly two vertices in $U_{s+i}$. It follows that $\Pi$ satisfies the statement of Observation~\ref{obs:sp} and is a failed power dominating partition.
Finally, note that for each $i\in [m]$, $\gamma_P(K_i) = s+1$. It follows that $m+s = \sum_{i=1}^m \gamma_P(K_i)-(m-1)\abs{C}$.
\end{proof}

We now provide a generalization of the upper bound in Theorem~\ref{thm:lowercut}.

\begin{observation}\label{thm:cutgen}
Let $G$ be a graph and $C$ a cut-set of $G$ of size $s$, with $G-C = H_1 \cup \cdots \cup H_m$ where each $H_i$ is a component of $G-C$.  In addition, for each $i \in [m]$, let $C_i$ be any subset of $C$ and let $L_i$ be the graph induced by $V(H_i) \cup C_i$. Then
\[\gamma_P(G) \leq \sum_{i=1}^m\gamma_P(L_i)+s.\]
\end{observation}

\begin{proof}
For each $i \in [m]=\{1,2,3...,m\}$, let $B_i$ be a minimum power dominating set of $L_i$. One can easily verify that $C \cup \bigcup_{i=1}^m B_i$ is a power dominating set of $G$.
\end{proof}

 One can easily construct a class of graphs similar to those in Proposition~\ref{constructcut1} that will show the upper bound given in Observation~\ref{thm:cutgen} is sharp. Perhaps a more interesting question is are there graphs for which the upper bound provided in Observation  \ref{thm:cutgen} is strictly less than the upper bound provided in Theorem~\ref{thm:lowercut}? The answer is yes. Indeed, consider a graph constructed as follows. For $i\in [2]$, let $H_i$ be the graph obtained from two disjoint copies of $K_{1, n}$ for $n \ge 3$ by adding an edge between the centers of each $K_{1, n}$. Construct $G$ from $H_1 \cup H_2 \cup \{x_1, x_2\}$ whereby each $x_i$ for $i \in [2]$ is adjacent to all vertices of $H_1 \cup H_2$. $\{x_1, x_2\}$ is certainly a cut-set of $G$ and if we let $C_1 = \{x_1\}$ and $C_2 = \{x_2\}$, then $L_1$ is the graph induced by $V(H_1) \cup \{x_1\}$ and $L_2$ is the graph induced by $V(H_2) \cup \{x_2\}$. Thus, $\gamma_P(L_i) =1$ for $i\in [2]$. It follows that the upper bound in Observation~\ref{thm:cutgen} yields $\sum_{i=1}^2\gamma_P(L_i) + 2 = 4$. On the other hand, the upper bound in Theorem~\ref{thm:lowercut} gives us $\sum_{i=1}^2\gamma_P(H_i) + 2 = 6$. 

\section{Conclusions}
In this paper, we provided new lower bounds on the power domination number of the Cartesian product of two graphs. Unlike previous lower bounds, this new bound depends on both factors. In addition, for graphs in certain classes, including trees, we prove that a Vizing-like bound holds with respect to the power domination number. Of course, the ultimate goal would be to show that for any pair of graphs $G$ and $H$, $\gamma_P(G\Box H) \ge \gamma_P(G)\gamma_P(H)$.

\end{document}